\documentclass[11pt]{article}
\usepackage[margin=1.5in]{geometry}
\usepackage{authblk, amsfonts, amsmath, amssymb, amsthm, bbm, xcolor, mathrsfs}
\usepackage{graphicx} 
\usepackage{graphicx, tikz}
\usetikzlibrary{decorations.markings}
\usetikzlibrary{arrows.meta}
\usepackage{enumitem} 
\usepackage{hyperref} 
\usepackage[textwidth=1.5in]{todonotes} 
\usepackage{thm-restate}

\newtheorem{thm}{Theorem}[section]

\newtheorem{corollary}[thm]{Corollary}

\newtheorem{lemma}[thm]{Lemma}

\theoremstyle{definition}

\newcommand{\C}{\mathcal{C}}
\DeclareMathOperator{\val}{val}

\definecolor{DeepCarrotOrange}{rgb}{0.91, 0.41, 0.17}
\definecolor{BananaYellow}{rgb}{1.0, 0.88, 0.21}

\definecolor{CornflowerBlue}{rgb}{0.39, 0.58, 0.93}
\definecolor{Magenta}{rgb}{0.50, 0.0, 0.50}
\definecolor{AppleGreen}{rgb}{0.55, 0.71, 0.0}
\definecolor{AO}{rgb}{0.0, 0.5, 0.0}

\setlist[enumerate]{itemsep=0em}
\tikzset{MyNode/.style={circle, draw, inner sep=2,outer sep=0, fill=gray}}
\tikzset{MyRedArc/.style={line width=2.3pt, red}}
\tikzset{MyBlueArc/.style={line width=2.3pt, blue}}
\tikzset{
  arrHalfway/.style={
    postaction={
      very thick,
      decorate,
      decoration={
        markings,
        mark= at position .56 with {\arrow{>}},
}}}}
\tikzset{
  arrHalfwayLoop/.style={
    postaction={
      very thick,
      decorate,
      decoration={
        markings,
        mark= at position .5 with {\arrow{<}},
}}}}

\title{The structure of group-labeled graphs\\ forbidding an immersion}
\author{Rose McCarty\thanks{Supported by the National Science Foundation under Grant No.~DMS-2452111.}}
\affil{Schools of Mathematics and Computer Science, Georgia Institute of Technology, USA.}
\author{Caleb McFarland\thanks{Supported in part by the Georgia Tech ARC-ACO Fellowship and in part by the National Science Foundation under Grant No.~DMS-2452111.}}
\affil{School of Mathematics, Georgia Institute of Technology, USA.}
\author{Paul Wollan}
\affil{Sapienza Universit\`{a} di Roma, Italy.}

\date{February 2025}

\begin{document}

\maketitle

\begin{abstract}
    A \emph{$\Gamma$-labeled} graph is an oriented graph with edges invertibly labeled by a group $\Gamma$. We prove a structure theorem for $\Gamma$-labeled graphs which forbid a fixed $\Gamma$-labeled graph as an immersion, for any finite $\Gamma$. Roughly, we show that such graphs admit a tree-cut decomposition in which every bag either contains few high degree vertices or is nearly signed over a proper subgroup of $\Gamma$. 
\end{abstract}

\section{Introduction}

Given a group $\Gamma$, a \emph{$\Gamma$-labeled graph} (also called a \emph{$\Gamma$-gain graph} or a \emph{$\Gamma$-voltage graph}) is an undirected graph whose oriented edges are invertibly labeled in $\Gamma$. That is, each edge of the graph has two orientations, and these orientations are associated with a group element $\alpha$ and its inverse $\alpha^{-1}$, respectively. Group-labeled graphs have been used to study graph embeddings and coverings~\cite{Gross74, GrossTucker77, GrossTucker01}, symmetry-forced rigidity~\cite{Bernstein22, KST21, Tanigawa15} (i.e., bar and join frameworks where every non-rigid motion breaks a specified symmetry), and the structural~\cite{GGW14, GerardsGraphicMatroids, SeymourGraphicMatroids} and extremal~\cite{NW22, PSWX24} properties of matroids. It is also expected that group-labeled graphs will play a key role in understanding the structure of graphs with a forbidden vertex-minor~\cite{BouchetCircleChar, McCartyThesis, decomposingSigned}.


In this paper, we prove a structure theorem for $\Gamma$-labeled graphs which forbid a fixed $\Gamma$-labeled graph as an immersion, for any finite group $\Gamma$. Roughly, an \emph{immersion} of a graph $H$ into a graph $G$ maps vertices of $H$ injectively to vertices of $G$, and maps edges of $H$ to edge-disjoint trails of $G$ with the appropriate ends. If $H$ and $G$ have associated $\Gamma$-labelings, then we require the label of each edge of $H$ to match the label of the corresponding trail in $G$. (A \emph{trail} is a sequence of oriented edges which yields a walk without repeated edges in the underlying undirected graph. Thus there is a corresponding word over $\Gamma$, and its value is the \emph{label} of the trail.) Since group-labeled graphs are usually considered up to an equivalence relation called shifting, we allow shifting in $G$ (and as a consequence, also in $H$) as well when defining immersions. See Section~\ref{sec:prelim} for precise definitions.

Our structure theorem roughly says the following; see Theorem~\ref{thm:MainStructureThm} for the full formal statement. It is convenient to require the graph to be $2$-edge-connected, but the theorem below may be applied separately to each edge-block of a graph.

\begin{thm}[Informal Statement]
\label{thm:informalStructure}
For any finite group $\Gamma$ and any $\Gamma$-labeled graph $(H, \gamma_H)$, there exists an integer $t$ so that for any $2$-edge-connected $\Gamma$-labeled graph $(G, \gamma)$ which forbids an immersion of $(H, \gamma_H)$, there exists a shifting $\gamma'$ of $\gamma$ so that $(G, \gamma')$ ``decomposes along small edge-cuts into pieces'' so that each piece either:\begin{enumerate}
\item has at most $t$ vertices of degree more than $t$, or
\item has at most $t$ edges labeled outside of a proper subgroup $\Gamma'$ of~$\Gamma$.
\end{enumerate}
\end{thm}

We also prove a rough converse to this theorem in Lemma~\ref{lem:convOfStructure}. That is, we prove that any $\Gamma$-labeled graph with such a decomposition forbids some (possibly larger) $\Gamma$-labeled graph as an immersion. We note that if $(H, \gamma_H)$ has $n$ vertices and each vertex is incident to at most $k$ edges, then our integer $t$ is $4kn|\Gamma|^{6+\lfloor \log_2(|\Gamma|)\rfloor}$. So for any fixed group this is $\mathcal{O}(nk)$. We leave it as an open problem whether there is a structure theorem whose parameters are polynomial in $k$, $n$, and $|\Gamma|$. The key case actually seems to be when $H$ has only one vertex.



A direct corollary of Theorem~\ref{thm:informalStructure} is that every $\Gamma$-labeled graph which forbids an $(H, \gamma_H)$ immersion and has sufficiently high edge-connectivity and number of vertices has (after possibly shifting) a bounded number of edges labeled outside of a proper subgroup $\Gamma'$ of $\Gamma$. This key outcome has also appeared in other work on the structure of $\Gamma$-labeled graphs. First of all, our theorem generalizes the theorem of Churchley and Mohar~\cite{ChurchleyMohar2018} for \emph{odd immersions}; this is the setting where $\Gamma=\mathbb{Z}_2$ and every edge is assigned the non-identity label. In this setting, every graph which is ``almost'' labeled over a proper subgroup of $\mathbb{Z}_2$ (as in the second outcome of the theorem) is actually the union of a bounded number of edges and a bipartite graph. So there is hope that any theorem which holds for bipartite graphs might ``approximately'' hold for graphs with a forbidden odd immersion.


The chromatic number is one example of how properties for bipartite graphs can be ``approximately" extended to graphs excluding a forbidding odd immersion.  Bipartite graphs trivially have bounded chromatic number.  The second author has built on such structural arguments to prove that every graph without $K_t$ as an odd immersion is $\mathcal{O}(t)$-colorable~\cite{mcfarland2025}. (We believe that directly applying the general structural results would yield an $\mathcal{O}(t^2)$ bound.) Overall, there is a rich literature showing that graphs with forbidden ``odd'' substructures tend to have small chromatic number; see for instance~\cite{Geelen2009oddMinor, LiuOddSubdivisions, Steiner22, SteinerSubdivisions, ThomassenSubdivisions} and the recent proof that Odd Hadwiger's Conjecture is actually false~\cite{KSSW25}. The development of efficient algorithms provides another area in which this approach can be applied to graphs excluding some ``odd" substructure. Intuitively, this is because in a bipartite graph: 1) one can easily tell whether two vertices are joined by an odd or even length path and 2) many integer programming problems can be solved in polynomial time. For examples of efficient algorithms in classes of graphs excluding ``odd" substructures see~\cite{CGKMW2025, FJWY2025, HuynhThesis, KRW2011, LiuYoo25}.


So in general, structure theorems for unlabeled graphs have broad applications to other areas such as graph coloring and algorithms and complexity. The difficulty for $\Gamma$-labeled graphs is that it is not clear which problems become easier when the graph is labeled over a proper subgroup of $\Gamma$ (except in the previously mentioned odd immersions case, where this means that the graph is ``almost'' bipartite). This question arises when studying \emph{any} notion of graph containment for $\Gamma$-labeled graphs. So it seems like an interesting direction for future research. Moreover, this setting serves as a midway point when trying to generalize work on graphs to binary matroids (and beyond). We note that the frame matroids of graphs labeled over the multiplicative group of a field $\mathbb{F}$ are themselves $\mathbb{F}$-representable~\cite{Zaslavsky03}. We refer the reader to work of Zaslavsky~\cite{zaslavskyI, Zaslavsky94} for more information about the interplay between graphs and their associated matroids.


There is another main example of a structure theorem for $\Gamma$-labeled graphs. As part of their matroid minors project with Whittle, Geelen and Gerards~\cite{GeelenGerardsGroupLabeled} proved a ``local structure theorem'' for $\Gamma$-labeled graphs with a forbidden minor when $\Gamma$ is finite and abelian. Their proof relies on a packing/covering duality by Chudnovsky, Geelen, Gerards, Goddyn, Lohman, and Seymour~\cite{chudnovsky2006packing}. This theorem describes the obstructions to packing a large collection of vertex-disjoint paths so that each path has a label which is not the identity element and has ends in a fixed set of vertices.

Our proof of Theorem~\ref{thm:informalStructure} follows a similar approach. In order to deal with the non-abelian case, we apply a more general packing-covering duality where the paths are labeled outside of a fixed subgroup $\Gamma'$ of $\Gamma$. Very general duality theorems of this flavor are proven in~\cite{GeelenRodriguezAPaths, PapAPaths, YamaguchiAPaths}; we apply the theorem of Pap~\cite{PapAPaths}. (Note that when $\Gamma'$ is a normal subgroup of $\Gamma$ such as in the abelian case, the theorem of Chudnovsky et al.~\cite{chudnovsky2006packing} can be applied to a quotient labeling in $\Gamma/\Gamma'$.) 

We also build on work for (unlabeled) graphs with a forbidden immersion. The third author~\cite{WOLLAN2015}, and independently DeVos, McDonald, Mohar, and Scheide~\cite{Devos2013note}, gave a structure theorem in this setting. We use some of the tools from this first paper. In particular, we use \emph{tree-cut decompositions}, which show how to decompose a graph in a tree-like way along edge cuts. For immersions of unlabeled graphs (or equivalently, when the group only has one element), the structure is similar to Theorem~\ref{thm:informalStructure}, except the second outcome goes away; every piece has few high degree vertices. Our key new contributions include Lemma~\ref{lem:uncrossingFlower} which lets us ``disentangle'' two carefully chosen immersions, and Lemmas~\ref{lem:LaminarCertificates} and~\ref{lem:decompFromContainers} which let us ``uncross edge cuts'' and combine many different ``local structures'' into a single global structure.

We note that there is a long history of studying the structure of group-labeled graphs, and in particular of studying approximate packing/covering dualities; see for instance~\cite{HJW2019, KK2016}. (These approximate packing/covering dualities are known as \emph{Erd\H{o}s-P\'{o}sa properties}.) Finally, there is also work studying the structure of group-labeled graphs where both orientations correspond to the same group value; see~\cite{LiuYoo25, ThomasYoo23, Wollan2010} for more information about this setting.

This paper is organized as follows. In Section~\ref{sec:prelim} we discuss our notation and the basics of group-labeled graphs and immersions, and then we state the structure theorem (Theorem~\ref{thm:MainStructureThm}). In Section~\ref{sec:packing} we deduce a packing/covering duality for edge-disjoint circuits hitting a fixed vertex from the vertex-disjoint version. In Section~\ref{sec:richFlower} we prove our ``local structure theorem'', Theorem~\ref{thm:enrichingFlower}. Finally, in Section~\ref{sec:strThm}, we prove the structure theorem as well as its rough converse, Lemma~\ref{lem:convOfStructure}.

\section{Preliminaries}
\label{sec:prelim}

We use standard graph theoretic notation. Given a graph $G$ and a set of vertices $B$, we write $\delta_G(B)$ (or sometimes just $\delta(B)$) for the set of all edges of $G$ with exactly one end in $B$. If $B$ consists of just one vertex $v$, then we write $\delta(v)$. We also write $E(B)$ for the set of all edges with both ends in $B$. If $X$ is a set of edges, a set of vertices, or a single edge or vertex of $G$, then we write $G\setminus X$ for the graph obtained from $G$ by deleting the vertices/edges in $X$.

All graphs are undirected and may have loops and parallel edges. For the group labeling, however, the orientation of an edge matters; traversing an edge in one direction corresponds to a group element $\alpha$, while traversing it in the other direction corresponds to the inverse $\alpha^{-1}$. So we now introduce some notation to deal with orientations. Given a graph $G$, each edge $e$ has two distinct orientations which we denote by $\vec{e}$ and $\vec{e}^{-1}$. Each oriented edge $\vec{e}$ has a \emph{tail} and a \emph{head}, and we say that $\vec{e}$ is oriented \emph{from} its tail \emph{to} its head. We call $e$ the \textit{underlying} edge of an oriented edge $\vec{e}$. We write $\vec{E}(G)$ for the set of all orientations of edges in $G$. 



Let $\Gamma$ be a group. We denote the group action multiplicatively and the identity by 1. A \emph{$\Gamma$-labeling} of a graph $G$ is a function $\gamma: \vec{E}(G) \rightarrow \Gamma$ such that $\gamma(\vec{e}^{-1}) = \gamma(\vec{e})^{-1}$ for any edge $e$ of $G$. We call $\gamma(\vec{e})$ the \textit{label} of an oriented edge $\vec{e}$. A \textit{$\Gamma$-labeled graph} is a pair $(G, \gamma)$ where $G$ is a graph and $\gamma$ is a $\Gamma$-labeling of $G$. There is an equivalence relation on $\Gamma$-labelings defined as follows. Given a $\Gamma$-labeled graph $(G,\gamma)$, a vertex $v \in V(G)$, and a group element $\alpha \in \Gamma$, \textit{shifting by $\alpha$ at $v$} means to append $\alpha$ at the beginning of the label of each oriented edge with tail $v$ and to append $\alpha^{-1}$ at the end of the label of each oriented edge with head $v$. Thus, if $\vec{e}$ is an oriented loop at $v$, then its new label after shifting is $\alpha\gamma(\vec{e})\alpha^{-1}$. Likewise, if $\vec{e}$ is an oriented edge whose tail is $v$ and whose head is not $v$, then its new label is $\alpha\gamma(\vec{e})$. This is the inverse of $\gamma(\vec{e}^{-1})\alpha^{-1}$, which is the new label of $\vec{e}^{-1}$. We say that a group-labeling $\gamma'$ of $G$ is a \emph{shifting} of $\gamma$ if it can be obtained by a (finite) sequence of such shiftings. 


A \textit{trail} $T$ in a graph $G$ is a sequence of oriented edges $\vec{e_1}\vec{e_2}\dots\vec{e_n}$ such that the underlying edges are all distinct, and for each $i \in [n-1]$, the head of $\vec{e_i}$ equals the tail of $\vec{e}_{i+1}$. Note that the endpoints of the edges need not be distinct; the same vertex may be visited multiple times. The  \textit{tail} of $T$ is the tail of $\vec{e}_1$ and the \textit{head} of $T$ is the head of $\vec{e_n}$. We call the head and tail of $T$ its \textit{ends}. Given a group labeling $\gamma$ of $G$, the \emph{label} of $T$ is $\gamma(T) = \gamma(\vec{e_1})\gamma(\vec{e_2})\ldots\gamma(\vec{e_n})$. The \emph{inverse} of $T$ is the trail $T^{-1}$ which is the sequence of edges in the reverse order and with the reverse orientation. Thus $\gamma(T^{-1})=\gamma(T)^{-1}$. The \textit{transitions} of $T$ are the ordered pairs $(\vec{e_i}, \vec{e_{i+1}})$ for $i \in [n-1]$. Thus $T$ has $n-1$ transitions. We view transitions of the form $(\vec{e_1}, \vec{e_{2}})$ and $(\vec{e_2}^{-1}, \vec{e_{1}}^{-1})$ as being equivalent; thus $T^{-1}$ has the same transitions as $T$. 




\subsection{Immersions}
\label{subsec:Immersions}

We refer the reader to Figure~\ref{fig:immersion} for a depiction of an immersion of a graph $H$ into a graph $G$. We now define immersions of group-labeled graphs; we define immersions of unlabeled graphs the same way, except we do not care about the labels of edges.

Let $\Gamma$ be a group. We say that a $\Gamma$-labeled graph $(G, \gamma_G)$ admits an \textit{immersion} of a $\Gamma$-labeled graph $(H, \gamma_H)$ if there exist a shifting $\gamma_G'$ of $\gamma_G$ and a function $\varphi$ which injectively maps $V(H)$ to $V(G)$ and which maps $\vec{E}(H)$ to trails in $G$ such that\begin{enumerate}
    \item for any edge $e$ of $H$, the trails $\varphi(\vec{e})$ and $\varphi(\vec{e}^{-1})$ are inverses,
    \item for any distinct edges $e_1$ and $e_2$ of $H$, the trails $\varphi(\vec{e_1})$ and $\varphi(\vec{e_2})$ are edge-disjoint, and
    \item for any oriented edge $\vec{e}$ of $H$, the tail of $\vec{e}$ is mapped to the tail of $\varphi(\vec{e})$, the head of $\vec{e}$ is mapped to the head of $\varphi(\vec{e})$, and $\gamma_H(\vec{e}) = \gamma_G'(\varphi(\vec{e}))$.
\end{enumerate} We say that $\varphi$ is an \emph{immersion} of $(H, \gamma_H)$ \emph{into} $(G, \gamma_G)$. We call the image of $V(H)$ under $\varphi$ the \textit{branch vertices} and the image of $\vec{E}(H)$ under $\varphi$ the \textit{branch trails}. The \textit{transitions} of an immersion are the transitions of its branch trails. We note that the definition above is often referred to as a \emph{weak} immersion in the literature, whereas \emph{strong} immersions require that no ``internal'' vertex of a branch trail is also a branch vertex; see~\cite{DvorakWollanStrong16, MarxWollanStrong14, WOLLAN2015} for the precise definitions.

\begin{figure}
\centering
\begin{tikzpicture}[scale =1, every node/.style={MyNode}]
    \node[draw=none, fill=none] (Y0) at (0, -1) {};
    \node[draw=none, fill=none] (Y0) at (0, 2.25) {};
    \node[rectangle, draw=none, fill=none]  at (0, 1.5) {$H$};
    \node[fill=black, inner sep=3, label={[below, yshift=-.3cm] $u$}] (X0) at (-1, 0) {};
    \node[fill=black, inner sep=3,  label={[below, yshift=-.3cm] $v$}] (X1) at (1, 0) {};
    \draw[MyBlueArc, arrHalfway] (X0) to (X1);
    \draw[MyRedArc, arrHalfwayLoop] (X1) to [out=65,in=115,looseness=30] (X1);
\end{tikzpicture}\hskip 1.75cm%
\begin{tikzpicture}[scale=1, every node/.style={MyNode}]
    \node[draw=none, fill=none] (Y0) at (0, -1) {};
    \node[draw=none, fill=none] (Y0) at (0, 2.25) {};
    \node[rectangle, draw=none, fill=none] (center) at (0,.5) {\Large$\longrightarrow$};
\end{tikzpicture}\hskip 2cm%
\begin{tikzpicture}[scale =1, every node/.style={MyNode}]
    \node[draw=none, fill=none] (Y0) at (0, -1) {};
    \node[draw=none, fill=none] (Y0) at (0, 2.25) {};
    \node[rectangle, draw=none, fill=none]  at (0, 2.25) {$G$};
    \node[fill=black, inner sep=3,  label={[below, yshift=-.15cm] $\varphi(u)$}] (Y0) at (-1, 0) {};
    \node[fill=black, inner sep=3,  label={[below, yshift=-.15cm] $\varphi(v)$}] (Y1) at (1, 0) {};
    \node (Y2) at (1, 1.75) {};
    \node (Y4) at (-1, 1.75) {};
    \draw[MyRedArc, arrHalfway] (Y1) to [bend left=18] (Y0);
    \draw[MyRedArc, arrHalfway] (Y0) -- (Y2);
    \draw[MyRedArc, arrHalfway] (Y2) -- (Y1);
    \draw[MyBlueArc, arrHalfway] (Y0) to [bend left=30] (Y4);
    \draw[MyBlueArc, arrHalfway] (Y4) to [bend left=30] (Y0);
    \draw[MyBlueArc, arrHalfway] (Y0) to [bend left=18] (Y1);
    \draw[thick] (Y0) to (Y4);
    \draw[thick] (Y2) -- (Y4);
\end{tikzpicture}
\vspace{-.5cm}
\caption{A depiction of an immersion $\varphi$ of $H$ into $G$.}
\label{fig:immersion}
\end{figure}
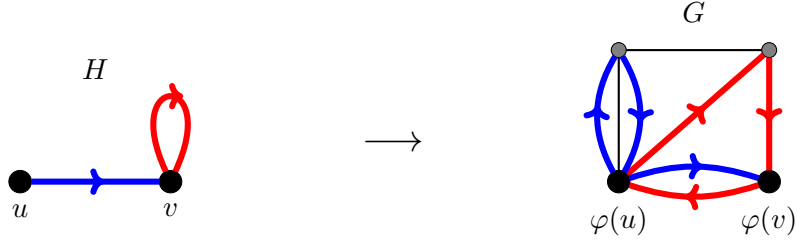


Equivalently, $(G, \gamma_G)$ admits $(H, \gamma_H)$ as an immersion if there exists a shifting $\gamma_G'$ of $\gamma_G$ so that $(H, \gamma_H)$ can be formed from a subgraph of $(G, \gamma_G')$ via repeatedly applying the following operation. Let $\vec{e_1}$ and $\vec{e_2}$ be oriented edges so that the head of $\vec{e_1}$ equals the tail of $\vec{e_2}$. Then to \emph{split off} $(\vec{e_1}, \vec{e_2})$, we delete the underlying edges $e_1$ and $e_2$ and add a new edge $e'$ whose ends are the tail of $\vec{e_1}$ and the head of $\vec{e_2}$; if $e'$ is oriented from the tail of $\vec{e_1}$ to the head of $\vec{e_2}$, then its label is the product of the labels of $\vec{e_1}$ and $\vec{e_2}$ (in that order). Note that splitting off $(\vec{e_1}, \vec{e_2})$ results in the same $\Gamma$-labeled graph as splitting off $(\vec{e_2}^{-1}, \vec{e_1}^{-1})$. Moreover, if $\varphi$ is an immersion of $(H, \gamma_H)$ into $(G, \gamma_G)$, then there exists a shifting $\gamma_G'$ of $\gamma_G$ so that $(H, \gamma_H)$ can be formed from a subgraph of $(G, \gamma_G')$ by splitting off the transitions of~$\varphi$.


We note that all of the definitions above also apply when $(H, \gamma_H)$ is labeled over a proper subgroup $\Gamma'$ of $\Gamma$, by just considering $(H, \gamma_H)$ to be $\Gamma$-labeled. The definitions are all well-behaved under shifting in $(H, \gamma_H)$; we just need to shift at the corresponding branch vertices of $\varphi$ in $(G, \gamma_G)$. Thus ``admitting an immersion'' is a transitive relation; if $(G, \gamma_G)$ admits an immersion of $(H, \gamma_H)$ and $(H, \gamma_H)$ admits an immersion of $(I, \gamma_I)$, then $(G, \gamma_G)$ also admits an immersion of~$(I, \gamma_I)$. 

\subsection{The structure theorem}
\label{subsec:structureThm}

We conclude this section by giving a precise statement of the structure theorem. First of all, a \textit{tree-cut decomposition} of a graph $G$ is a pair $(T, \mathscr{B})$ such that $T$ is a tree and $\mathscr{B} = (B_t : t \in V(T))$ is a collection of pairwise-disjoint (possibly empty) subsets of $V(G)$ whose union is $V(G)$. We call the sets in $\mathscr{B}$ the \textit{bags} of the decomposition. Roughly, the structure theorem says that each bag either 1) has a ``torso'' with few high degree vertices, or 2) has ``almost all of its edges'' labeled over a proper subgroup. We now state these outcomes more precisely.

So, let $G$ be a graph, and let $(T, \mathscr{B})$ be a tree-cut decomposition of $G$. Given a bag $B_t \in \mathscr{B}$, the \textit{torso} of $B$ is the graph $H$ obtained from $G$ as follows. Let $T_1, \dots, T_\ell$ be the components of $T \setminus t$. For each $i \in [\ell]$, let $V_i = \bigcup_{s \in V(T_i)} B_{s}$. Then $H$ is created by identifying each set $V_i$ to a single vertex, which we call a \emph{new vertex} of the torso. We keep parallel edges, but we remove any loops at the new vertices. 

Now, let $\Gamma$ be a group, let $(G, \gamma)$ be a $\Gamma$-labeled graph, and let $B\subseteq V(G)$. Then a \textit{certificate for $B$} is a tuple $(X,\gamma')$ so that $X$ is a set of edges of $G$ with $\delta(B) \subseteq X \subseteq \delta(B) \cup E(B)$ and $\gamma'$ is a shifting of $\gamma$ so that the $\gamma'$-labels of the edges in $E(B)\setminus X$ generate a proper subgroup of~$\Gamma$. The \emph{value} of a certificate is
\begin{align*}
    \val(X,\gamma') &= |\delta(B)| + 2|X \cap E(B)|.
\end{align*}This counts the edges of $X$, except edges with both ends in $B$ are counted twice. 

We are now ready to state the structure theorem. We state it for $2$-edge-connected graphs just to simplify things; a version for general graphs can be obtained by applying the theorem to each edge-block. We note that the \emph{degree} of a vertex $v$ is considered to be $|\delta(v)|$; we do not count loops but we do count parallel edges.

\begin{restatable}{thm}{maintheorem}
\label{thm:MainStructureThm}
    For any finite group $\Gamma$ and positive integers $k$ and $n$, there is an integer $t=4kn|\Gamma|^{6 + \lfloor\log_2|\Gamma|\rfloor}$ so that the following holds for any 2-edge-connected $\Gamma$-labeled graph $(G, \gamma)$ which forbids an immersion of some $n$-vertex $\Gamma$-labeled graph where each vertex is incident to at most $k$ edges. 
    
    There exist a shifting $\gamma'$ of $\gamma$ and a tree-cut decomposition $(T, \mathscr{B})$ of $G$ such that for every bag $B \in \mathscr{B}$, at least one of the following holds.
    \begin{enumerate}
        \item There are at most $n|\Gamma|$ vertices in the torso of $B$ of degree greater than $t$, and none of them are the new vertices of the torso.
        \item There exists $X \subseteq E(G)$ so that $(X,\gamma')$ is a certificate for $B$ of value at most~$t$.
    \end{enumerate}
\end{restatable}
\noindent We note that the two outcomes guarantee that the tree-cut decomposition has bounded \emph{adhesion} (this is the largest edge-cut of $G$ corresponding to an edge of $T$; see \cite{WOLLAN2015} for a formal definition).

We also prove that an approximate converse of Theorem~\ref{thm:MainStructureThm} holds in Lemma~\ref{lem:convOfStructure}. That is, we prove that any graph with a tree-cut decomposition as in Theorem~\ref{thm:MainStructureThm} also forbids some $\Gamma$-labeled graph as an immersion.

\section{Packing circuits labeled outside of a subgroup}
\label{sec:packing}

In this section we obtain an Erd\H{o}s-P\'{o}sa property for edge-disjoint circuits which begin at a fixed vertex and are labeled outside of a fixed subgroup $\Gamma'$. Informally, a collection of objects ``has the Erd\H{o}s-P\'{o}sa property'' if they satisfy a rough duality between packing and covering. This property is named after the theorem of Erd\H{o}s and P\'{o}sa~\cite{EP1965} that every graph without $k$ vertex-disjoint cycles has a set of at most $\mathcal{O}(k \log k)$ vertices whose deletion yields an acyclic graph. We obtain our desired Erd\H{o}s-P\'{o}sa property as a corollary of a min-max theorem for ``non-returning'' vertex-disjoint paths by Pap~\cite{PapAPaths}, which reduces this setting to a problem of Gallai~\cite{Gallai}.

First of all, a \textit{circuit} is a trail whose head and tail are equal; we say that it \emph{begins} at that vertex. Let us consider how the label of a circuit $C$ changes under shifting. If $C$ begins at a vertex $x$, then shifting at any vertex besides $x$ does not change the label of $C$. Shifting at $x$ by a group element $\alpha$ conjugates the label of $C$ by $\alpha$. It is often convenient to change the starting vertex of a circuit. So if $C= \vec{e_1}\vec{e_2} \ldots \vec{e_n}$ is a circuit, then any circuit of the form $C'=\vec{e_i}\vec{e_{i+1}} \ldots \vec{e_n} \vec{e_1} \ldots \vec{e_{i-1}}$ is called a \emph{cyclic reordering} of $C$. Note that the labels of $C$ and $C'$ are also conjugate because one is of the form $\alpha\beta$ and the other is of the form $\beta\alpha$. We view paths as trails without repeated vertices; thus the \emph{label} of a path is the corresponding group element, and paths have a fixed beginning and end.


First we observe that the following result is a corollary of a theorem of Pap~\cite{PapAPaths}. We then obtain the desired result in Corollary~\ref{cor:EP}.


\begin{thm}[{\cite[Theorem~2.1]{PapAPaths}}]\label{thm:NonzeroAPathsEP}
    Let $\Gamma$ be a group, let $\Gamma'$ be a subgroup of $\Gamma$, and let $r\in \mathbb{Z}^+$. Then for any $\Gamma$-labeled graph $(G, \gamma)$ and set $A \subseteq V(G)$, at least one of the following holds.
    \begin{enumerate}
        \item There exist a collection of $r$ pairwise vertex-disjoint paths so that each path has both ends in~$A$ and has label outside of~$\Gamma'$.
        \item There exists a set $X \subseteq V(G)$ of size at most $2r-2$ so that $G \setminus X$ has no path which has both ends in $A$ and has label outside of~$\Gamma'$.
    \end{enumerate}
\end{thm}

To see that Theorem~\ref{thm:NonzeroAPathsEP} follows from~\cite{PapAPaths}, we let the ``set of potentials'' be $\Gamma$, and we let the ``potential of origin'' be $\Gamma'$ for every vertex in $A$. The theorem then tells us that the maximum number of vertex-disjoint paths with ends in $A$ and label outside of $\Gamma'$ is equal to the minimum over all ``balanced'' sets of edges $F$, of the maximum number of vertex-disjoint paths of $G-F$ whose ends are in $A \cup V(F)$. (We write $V(F)$ for the set of all vertices incident to an edge in $F$.) By a theorem of Gallai~\cite{Gallai}, either there exist $r$ such paths in $G-F$ (and therefore the first outcome of the theorem holds), or there exists a set $X$ of at most $2r-2$ vertices in $G-F$ so that $(G-F)\setminus X$ has no path with both ends in $A \cup V(F)$. In this latter case we again apply the min-max theorem of Pap~\cite{PapAPaths} to see that $G\setminus X$ has no path which has both ends in $A$ and has label outside of~$\Gamma'$, as desired.

We now derive the following analogous result for edge-disjoint circuits as a corollary of Theorem~\ref{thm:NonzeroAPathsEP}.

\begin{corollary}\label{cor:EP}
    Let $\Gamma$ be a group, let $\Gamma'$ be a subgroup of $\Gamma$, and let $r \in \mathbb{Z}^+$. Then for any $\Gamma$-labeled graph $(G, \gamma)$ and vertex $x$, at least one of the following holds.
    \begin{enumerate}
        \item There exist a collection of $r$ pairwise edge-disjoint circuits so that each circuit begins at $x$ and has label outside of~$\Gamma'$.
        \item There exist a set $X \subseteq E(G)$ of size at most $2r-2$ so that $G\setminus X$ has no circuit which begins at $x$ and has label outside of~$\Gamma'$.
    \end{enumerate}
\end{corollary}
\begin{proof}
    Let $H$ be the graph obtained from $G$ by ``separating'' each vertex $v$ of $G$ into a clique $K_v$ whose size is the number of edges incident to $v$ in $G$. Thus $H$ consists of pairwise vertex-disjoint cliques $(K_v: v \in V(G))$ and a perfect matching corresponding to $E(G)$. (Formally, $H$ is the line graph of the $1$-subdivision of $G$.) Let $\gamma_H$ be the $\Gamma$-labeling of $H$ where each oriented edge from the matching receives the same label as in $(G, \gamma)$, and the other edges (those in the cliques) receive label~$1$.



    Let $A$ be the vertices in $K_x$, and apply Theorem~\ref{thm:NonzeroAPathsEP} to $(H, \gamma_H)$. First suppose that there exist $r$ pairwise vertex-disjoint paths in $(H, \gamma_H)$ so that each path has both ends in $A$ and has label outside of $\Gamma'$. For each of these paths, consider the corresponding trail in $G$; as we walk along the path, each time a matching edge of $H$ is traversed, we append that oriented edge of $G$ to the trail. These trails are pairwise edge-disjoint, have labels outside of $\Gamma'$, and begin at $x$. So the first outcome of the corollary holds. 

    Thus we may assume that there exists a set $X \subseteq V(H)$ of size at most $2r-2$ so that $H\setminus X$ has no path which has both ends in $A$ and has label outside of $\Gamma'$. Let $X' \subseteq E(G)$ be the set of all edges of $G$ so that at least one end of the corresponding matching edge in $H$ is in $X$. Then $X'$ has size at most $2r-2$. Suppose towards a contradiction that $G\setminus X'$ has a circuit $C$ which begins at $x$ and has label outside of $\Gamma'$. Then $H\setminus X$ has a corresponding path with ends in $A$ and label outside of $\Gamma'$; the vertices of this path are the endpoints of the matching edges in $H$ corresponding to edges in $C$. This contradiction completes the proof.
\end{proof}

We need one final lemma to handle the second outcome of Corollary~\ref{cor:EP}. This lemma characterizes when there is no circuit which begins at a fixed vertex $x$ and is labeled outside of $\Gamma'$. An \textit{edge-block} of a graph is a maximal connected bridgeless subgraph. So an edge-block is either a single vertex or a maximal $2$-edge-connected subgraph. The edge-blocks of a graph partition its vertex set, so each vertex is in a unique edge-block.


\begin{lemma}
\label{lem:edgeBlock}
    Let $\Gamma$ be a group, let $\Gamma'$ be a subgroup of $\Gamma$, and let $(G, \gamma)$ be a $\Gamma$-labeled graph with a vertex $x$ so that there is no circuit which begins at $x$ and is labeled outside of $\Gamma'$. Then there exists a shifting $\gamma'$ of $\gamma$ obtained by shifting at vertices in $V(G)\setminus \{x\}$ so that every oriented edge $\vec{e}$ in the edge-block of $G$ containing $x$ has $\gamma'(\vec{e})\in \Gamma'$.
\end{lemma}
\begin{proof}
    Let $H$ be the edge-block of $G$ containing $x$. If $H$ is just $x$ then we are done. So we may assume that $H$ is $2$-edge-connected. Thus there exists a cycle $C$ in $H$ which contains $x$. We may shift at vertices in $V(G)\setminus \{x\}$ until all but one of the edges of $C$ is labeled by the identity. Since this shifting does not change the label of $C$, the other edge of $C$ is labeled by some element in $\Gamma'$. 
    
    Now, let $C=T_0, T_1, T_2, \ldots, T_k$ be an ear decomposition of a subgraph of $H$, and assume by induction that there is a shifting $\gamma'$ of $\gamma$ so that every edge in any of $T_0, T_1, T_2, \ldots, T_{k-1}$ is labeled in $\Gamma'$. That is, each $T_i$ for $i\in \{1,2,\ldots, k\}$ is either a path with two distinct ends both in $T_0, T_1, T_2, \ldots, T_{i-1}$, or a cycle which begins at a vertex in $T_0, T_1, T_2, \ldots, T_{i-1}$. In both cases, no internal vertex of $T_i$ (that is, no vertex of $T_i$ other than its ends) is in any of $T_0, T_1, T_2, \ldots, T_{i-1}$. Since $H$ is $2$-edge-connected, it has an ear decomposition; we can always add one additional path or cycle until all edges are covered. Thus it will complete the proof to show that there is a shifting $\gamma''$ of $\gamma'$ so that every edge in $T_0, T_1, T_2, \ldots, T_{k}$ is labeled in~$\Gamma'$.

    Let $\gamma''$ be the shifting of $\gamma'$ obtained by shifting at internal vertices of $T_k$ until at most one edge of $T_k$ has a non-identity label. Let $\alpha$ be the label of this final edge. We just need to show that $\alpha \in \Gamma'$. So, let $u$ and $v$ (possibly with $u=v$) be the ends of $T_k$. Since the union of $T_0, T_1, T_2, \ldots, T_{k-1}$ is $2$-edge-connected, we can find in this graph two edge-disjoint paths $R_u$ and $R_v$ so that their tails are $u$ and $v$, respectively, and their head is $x$. Then, up to symmetry between $u$ and $v$, we find that there exists a circuit which begins at $x$ and has label $\gamma''(R_u^{-1})\alpha \gamma''(R_v)$. Since this element is contained in $\Gamma'$ and so are the labels of $R_u$ and $R_v$, this means that $\alpha \in \Gamma'$, as desired.
\end{proof}
\section{Finding a rich flower immersion}
\label{sec:richFlower}

In this section, we first discuss the ``universal'' group-labeled graphs which we look for as immersions. Then we prove Theorem~\ref{thm:enrichingFlower}, which describes the ``local structure'' relative to an unlabeled flower immersion. For the following definitions, see Figure~\ref{fig:FlowerPetal}. 


For positive integers $k$ and $n$, we define the \textit{$(k,n)$-flower} to be the graph on $n+1$ vertices $x, y_1, y_2, ..., y_n$ so that for each $i \in [n]$, there are $k$ parallel edges between $x$ and $y_i$. We call $x$ the \textit{center vertex} of the flower. If $\varphi$ is an immersion of a flower into a graph $G$, then we call the branch vertex of $\varphi$ corresponding to the center of the flower the \emph{branch center}. A \emph{petal} is a subgraph of the flower consisting of $x$, some $y_i$, and the $k$ parallel edges between them. Note that if $H$ is a graph on $n$ vertices where each vertex is incident to at most $k$ edges, then the $(k,n)$-flower admits $H$ as an immersion; we map $V(H)$ to $\{y_1,y_2, \ldots, y_n\}$, and then we greedily map $\vec{E}(H)$ to two-edge trails of the flower.

\begin{figure}
\centering
\begin{tikzpicture}[scale=1, every node/.style={MyNode}]
    \def \rad {2}
    \def \start {-75}
    \foreach \i in {1, ..., 5}{
        \node (Y\i) at (\start+\i*55:\rad) {};
    }
    \node[label={[below, yshift=-.3cm] $x$}] (X) at (0,0) {};
    \foreach \i in {1, ..., 5}{
        \draw[thick] (X) to [bend left=30] (Y\i);
        \draw[thick] (X) to (Y\i);
        \draw[thick] (X) to [bend right=30] (Y\i);
    }
    \draw[MyRedArc] (X) to [bend left=30] (Y3);
    \draw[MyRedArc] (X) to (Y3);
    \draw[MyRedArc] (X) to [bend right=30] (Y3);
    \node[inner sep=3, fill=red] (X) at (0,0) {};
    \node[inner sep=3, fill=red] (Y3) at (\start+3*55:\rad) {};
\end{tikzpicture}
\caption{A $(3,5)$-flower with one petal highlighted in bold red.}
\label{fig:FlowerPetal}
\end{figure}
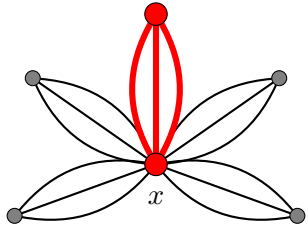

Let $\Gamma$ be a finite group. The \textit{$(\Gamma, k, n)$-rich flower} is the $\Gamma$-labeled graph where the graph is the $(|\Gamma|k, n)$-flower and, for all $i \in [n]$ and $\alpha \in \Gamma$, there exist $k$ oriented edges from $x$ to $y_i$ with label $\alpha$. Note that if $(H, \gamma_H)$ is a $\Gamma$-labeled graph on $n$ vertices where each vertex is incident to at most $k$ edges, then the $(\Gamma, k, n)$-rich flower admits $(H, \gamma_H)$ as an immersion. To see this, first map $V(H)$ to $\{y_1, y_2, \ldots, y_n\}$ as before. Then select one arbitrary orientation $\vec{e}$ of each edge $e$ of $H$, and map $\vec{e}$ to a two-edge trail which begins with an edge labeled by the identity and oriented towards the center, and ends with a $\gamma_H(\vec{e})$-labeled edge oriented away from the center.

Finally, we need one more type of flower in order to slowly build up a $(\Gamma, k, n)$-rich flower. When we actually find this type of flower as an immersion, it will be for larger and larger subgroups $\Gamma'$ of $\Gamma$. So, we say that a $\Gamma$-labeled graph is a \emph{$(\Gamma, k, n)$-generating flower} if there exists a generating set $S$ for $\Gamma$ so that $1 \in S$, the graph is an $(|S|k, n)$-flower, and for all $i \in [n]$ and $\alpha \in S$, there exist $k$ oriented edges from $x$ to $y_i$ with label $\alpha$. We call $S$ the \emph{generator}. 

Now we show that when $k$ is sufficiently large, every $(\Gamma, k, n)$-generating flower is ``universal''.

\begin{lemma}
\label{lem:genFlowUniversal}
    Let $\Gamma$ be a finite group and $k$ and $n$ be positive integers. Then any $(\Gamma, k|\Gamma|^2, n)$-generating flower admits an immersion of the $(\Gamma, k, n)$-rich flower.
\end{lemma}
\begin{proof}
    Let $(G, \gamma)$ be a $(\Gamma, k|\Gamma|^2, n)$-generating flower, and let $S$ be its generator.

    Consider the Cayley graph of $\Gamma$ with respect to $S$; this is the directed graph with vertex-set $\Gamma$ so that for every $g \in \Gamma$ and $\alpha \in S$, there is an arc from $g$ to $g\alpha$. Since $S$ is a generating set, every element of $\Gamma$ can be written as some word $\alpha_1\alpha_2 \ldots \alpha_\ell$ where $\alpha_1,\alpha_2, \ldots, \alpha_\ell \in S$. Then the corresponding arcs form a walk from the vertex $1$ to the group element $\alpha_1\alpha_2 \ldots \alpha_\ell$ that we started with. This walk contains a directed path, and so every $g \in \Gamma$ can be written as a word over $S$ of length at most~$|\Gamma|$.

    We form the immersion of the $(\Gamma, k, n)$-rich flower into $(G, \gamma)$ as follows. Fix $i \in [n]$ and consider the petal of $(G, \gamma)$ on $y_i$ and the center $x$. For each $g \in \Gamma$, we wish to use this petal to create $k$ pairwise edge-disjoint trails from $x$ to $y_i$ of label $g$. We now describe how to create one such trail; we then repeat the process $k$ total times to create the others. To create one trail of label~$1$, we just use one edge of label~$1$. So, suppose we wish to create a trail of label $g$ for some $g \in \Gamma\setminus \{1\}$. Let $\alpha_1,\alpha_2 \ldots, \alpha_\ell \in S \setminus \{1\}$ be such that $\ell \leq |\Gamma|$ and $\alpha_1\alpha_2 \ldots \alpha_\ell=g$. Begin the trail at the center $x$, use an oriented edge of label $\alpha_1$ to get to $y_i$, then an edge of label~$1$ to return to $x$, and repeat this process for $\alpha_2, \ldots, \alpha_{\ell-1}$. Then finally use one oriented edge of label $\alpha_\ell$ to return to~$y_i$. 
    
    There are enough oriented edges since, when creating $|\Gamma|$-many trails (with one trail for each element of $\Gamma$),  each $\alpha \in S$ is used at most $|\Gamma|^2$ times. 
\end{proof}


The rest of this section is dedicated to understanding the ``local structure'' relative to an unlabeled flower immersion. We aim to find an increasing sequence of subgroups $\{1\}=\Gamma_0 \subsetneq \Gamma_1 \subsetneq \ldots \subsetneq \Gamma_{\ell}=\Gamma$ so that our $\Gamma$-labeled graph admits an immersion of a $(\Gamma_i, k_i, n_i)$-rich flower for each $i \in [\ell]$, where $k_0, k_1, \ldots, k_{\ell}$ and $n_0, n_1, \ldots, n_{\ell}$ are decreasing (but not too quickly). For the starting point, we require the following Ramsey-type lemma for flowers in group-labeled graphs.


\begin{lemma}
\label{lem:ZeroFlower}
    Let $\Gamma$ be a finite group, let $k$ and $n$ be positive integers, and let $(G, \gamma)$ be a $\Gamma$-labeled graph so that $G$ is a $(k|\Gamma|, n)$-flower. Then there is an immersion $\varphi$ of the $(\{1\}, k, n)$-rich flower into $(G, \gamma)$ so that the branch center is the center of~$G$.
\end{lemma}
\begin{proof}
    Let $x$ be the center of $G$, and let $y_1, \dots, y_{n}$ be the other vertices of $G$. For each $i \in [n]$, let $\alpha_i \in \Gamma$ be the most common label of the edges oriented from $x$ to $y_i$. Note that there are at least $k$ such oriented edges of label $\alpha_i$. Now, let $\gamma'$ be the labeling obtained from $\gamma$ by, for each $i \in [n]$, shifting by $\alpha_i$ at $y_i$. (The shiftings may be performed in any order.) Thus, in $(G, \gamma')$, for each $i\in [n]$, there are at least $k$ edges between $x$ and $y_i$ of label~$1$. This yields the desired immersion; in fact the $(\{1\}, k, n)$-rich flower is a subgraph of $(G, \gamma')$.
\end{proof}

The following key lemma will be used to ``disentangle'' a collection of circuits found using Corollary~\ref{cor:EP} from a flower immersion. It is convenient to give a name to this collection of circuits. So, let $\Gamma$ be a group  and let $(G, \gamma)$ be a $\Gamma$-labeled graph. Given a vertex $x$ of $G$ and a subgroup $\Gamma'$ of $\Gamma$, a \emph{simple flower for $(x, \Gamma')$} is a collection $\C$ of pairwise edge-disjoint circuits so that each circuit in $\C$ begins at $x$ and is labeled outside of $\Gamma'$. 


We note that in the lemma below, when we count branch trails, we only count branch trails of the form $T$ and $T^{-1}$ once. When we apply the lemma, the graph $H$ will be a flower.


\begin{lemma}
\label{lem:uncrossingFlower}
    Let $\Gamma$ be a group, let $(G, \gamma)$ be a $\Gamma$-labeled graph, and let $\varphi$ be an immersion of a graph $H$ into $G$. Suppose that $H$ has a vertex which is incident to every edge of $H$, and let $x \in V(G)$ be the corresponding branch vertex of $\varphi$. Let $r$ be a positive integer and let $\Gamma'$ be a subgroup of $\Gamma$ so that $(G, \gamma)$ contains a simple flower $\C$ for $(x, \Gamma')$ of size $r$. Then we can choose $\C$ so that all but $2r$ of the branch trails of $\varphi$ are edge-disjoint from every circuit in~$\C$.
\end{lemma}
\begin{proof}
    We choose $\C$ to minimize the number of transitions of trails in $\C$ which are not transitions of $\varphi$. (It is important that every circuit in $\C$ begins at $x$; cyclically reordering a circuit changes one of its transitions.) We claim that every branch trail of $\varphi$ which does not contain the first or last edge of a circuit in $\C$ is edge-disjoint from every circuit in $\C$. Proving this claim will complete the proof of the lemma.

    Suppose otherwise, and let $T$ be a branch trail of $\varphi$ which is a counterexample. We may assume that $T$ is oriented to start from $x$. Let $\vec{e}$ be the first oriented edge of $T$ whose underlying edge $e$ is in a circuit $C \in \C$. By possibly replacing $C$ with $C^{-1}$, we may assume that $\vec{e}$ is in $C$. Let $T_1, T_2, C_1, C_2$ be trails so that $T = T_1 T_2$, $C = C_1 C_2$, and $\vec{e}$ is the first oriented edge of both $T_2$ and $C_2$. Note that by assumption $C_1$ is nonempty. Let $\vec{f}$ denote the last oriented edge of~$C_1$. Thus $(\vec{f}, \vec{e})$ is a transition of $C$ that is not a transition of~$\varphi$.

    First suppose that $T_1$ is empty. So the tail of $\vec{e}$ is $x$. Since $\gamma(C_1)\gamma(C_2) = \gamma(C) \not\in \Gamma'$, at least one of $C_1, C_2$ is labeled outside of $\Gamma'$. Both $C_1$ and $C_2$ are circuits which have $x$ as their end and do not contain the transition $(\vec{f}, \vec{e})$. Thus replacing $C$ with the appropriate choice of $C_1, C_2$ contradicts the choice of $\C$. 
    
    
    Now we may assume that $T_1$ is nonempty. Note that both $C_1 T_1^{-1}$ and $T_1 C_2$ are circuits which have $x$ as their end and do not contain the transition $(\vec{f}, \vec{e})$. Also, since $\gamma(C_1 T_1^{-1})\gamma(T_1 C_2) = \gamma(C) \not\in\Gamma'$, at least one of $C_1 T_1^{-1}$ and $T_1 C_2$ is labeled outside of $\Gamma'$. Note that every transition of $T_1 C_2$ is also a transition of $T$ or $C$, including the transition ``between'' $T_1$ and $C_2$. Thus we may assume that the other circuit $C_1 T_1^{-1}$ is the one labeled outside of $\Gamma'$. To handle this case, note that $C_2$ contains some transition which is not a transition of $\varphi$ because $T$ does not contain the last oriented edge of $C_2$. Therefore replacing $C$ with $C_1 T_1^{-1}$ contradicts our choice of~$\C$. This completes the proof of Lemma~\ref{lem:uncrossingFlower}.
\end{proof}

Now we prove the final lemma of this section; it describes when we can change a flower immersion to be labeled over a larger subgroup.


\begin{lemma}
\label{lem:LocalStructureStep}
    Let $\Gamma$ be a finite group, let $\Gamma'$ be a proper subgroup of $\Gamma$, and let $\varphi$ be an immersion of a $(\Gamma', 2k|\Gamma|, 2n)$-generating flower into a $\Gamma$-labeled graph $(G, \gamma)$. Then at least one of the following holds.
    \begin{enumerate}
        \item There exists a subgroup $\Gamma''$ of $\Gamma$ which properly contains $\Gamma'$ such that $(G, \gamma)$ admits an immersion of a $(\Gamma'', k, n)$-generating flower which has the same branch center as~$\varphi$.
        \item There exist a shifting $\gamma'$ of $\gamma$ and a set $X \subseteq E(G)$ of size at most $2nk|\Gamma|$ so that the edge-block of $G \setminus X$ containing the branch center of $\varphi$ is labeled over~$\Gamma'$.
    \end{enumerate}
\end{lemma}
\begin{proof}
    Since both of the desired outcomes are invariant under shifting $\gamma$, we may assume that we have already shifted so that each edge of the flower yields a branch trail of the correct label in $(G, \gamma)$. For convenience, let $x \in V(G)$ denote the branch center of $\varphi$ and set $r = nk(|\Gamma|-1)$. We apply Corollary~\ref{cor:EP}.
    
    First suppose that there exists a collection of $r$ pairwise edge-disjoint circuits so that each circuit begins at $x$ and has label outside of $\Gamma'$. This is a simple flower for $(x, \Gamma')$ of size $r$. Thus, by Lemma~\ref{lem:uncrossingFlower}, there exists a simple flower $\C$ of size $r$ so that all but $2r$ of the branch trails of $\varphi$ are edge-disjoint from every circuit in $\C$. 
    
    We then delete the $n$ petals with the largest number of branch trails sharing edges with circuits in $\C$. Each petal now has at most $2r/n$ branch trails which intersect circuits in $\C$. Thus after sacrificing $n$ petals and $2r/n=2k(|\Gamma|-1)$ branch paths on the remaining petals, the $r$ circuits in $\C$ are edge-disjoint from the remaining $(\Gamma', 2k, n)$-generating flower immersion. Thus, since $r=nk(|\Gamma|-1)$, we can take $nk$ of these circuits in $\C$ which have the same label $\alpha \in \Gamma \setminus \Gamma'$. Let $\Gamma''$ be the subgroup generated by $\Gamma' \cup \{\alpha\}$. We can further sacrifice $k$ of the trails with identity label on each petal to add $k$ trails with label $\alpha$ for each petal of the immersion, resulting in a $(\Gamma'', k, n)$-generating flower immersion with the same branch center as~$\varphi$.

    For the final case, suppose that the other outcome of Corollary~\ref{cor:EP} holds. So there exists a set $X \subseteq E(G)$ of size at most $2r\leq 2nk|\Gamma|$ so that $G\setminus X$ has no circuit which begins at $x$ and has label outside of $\Gamma'$. Then by Lemma~\ref{lem:edgeBlock}, there exists a shifting $\gamma'$ of $\gamma$ which labels every oriented edge in the edge-block of $G\setminus X$ containing $x$ within $\Gamma'$. Then the second outcome of the lemma holds.
\end{proof}

We are ready to prove the main result of this section. It is obtained by iteratively applying Lemma~\ref{lem:LocalStructureStep} until we have found a generating flower immersion labeled over as large a subgroup as possible. 

\begin{thm}
\label{thm:enrichingFlower}
    For any finite group $\Gamma$ and positive integers $k$ and $n$, there exist integers $k' = k|\Gamma|^{4 + \lfloor\log_2 |\Gamma|\rfloor}$ and $n' = n|\Gamma|$ so that for any $\Gamma$-labeled graph $(G, \gamma)$ and immersion $\varphi$ of the $(k', n')$-flower into $G$, at least one of the following holds.
    \begin{enumerate}
        \item The graph $(G, \gamma)$ admits an immersion of the $(\Gamma, k, n)$-rich flower which has the same branch center as~$\varphi$.
        \item\label{itm:EF2} There exist a proper subgroup $\Gamma'$ of $\Gamma$, a shifting $\gamma'$ of $\gamma$, and a set $X \subseteq E(G)$ of size at most $2n'k'|\Gamma|$ such that the edge-block of $G\setminus X$ containing the branch center of $\varphi$ is labeled over $\Gamma'$.
    \end{enumerate}
\end{thm}
\begin{proof}
    Suppose that the second outcome does not hold. For convenience, let $x \in V(G)$ denote the branch center of $\varphi$. By Lemma~\ref{lem:ZeroFlower} and the transitivity of immersions, $(G, \gamma)$ admits an immersion of the $(\{1\}, k|\Gamma|^3 \cdot |\Gamma|^{\lfloor\log_2 |\Gamma|\rfloor}, n|\Gamma|)$-rich flower with branch center $x$. Note that this also yields an immersion of the $(\{1\}, k |\Gamma|^2 \cdot (2|\Gamma|)^{\lfloor\log_2 |\Gamma|\rfloor}, n2^{\log_2|\Gamma|})$-generating flower. By repeatedly applying Lemma~\ref{lem:LocalStructureStep}, we obtain an immersion of a $(\Gamma, k|\Gamma|^2, n)$-generating flower through a sequence of subgroups $\{1\} =\Gamma_0 \subsetneq \Gamma_1 \subsetneq \dots \subsetneq \Gamma_\ell = \Gamma$. Note that $|\Gamma| = \prod_{i=1}^\ell |\Gamma_i| / |\Gamma_{i-1}| \geq 2^\ell$, so $\ell \leq \lfloor\log_2|\Gamma|\rfloor$ (in fact $\ell$ is at most the number of prime divisors of $|\Gamma|$). Finally, we obtain a $(\Gamma, k, n)$-rich flower as an immersion of our $(\Gamma, k|\Gamma|^2, n)$-generating flower by applying Lemma~\ref{lem:genFlowUniversal}. 
\end{proof}

\section{Proving the structure theorem}
\label{sec:strThm}

This section is dedicated to proving the main structure theorem (Theorem~\ref{thm:MainStructureThm}) and its partial converse. We begin by proving the partial converse stated below.

\begin{lemma}
\label{lem:convOfStructure}
    Let $\Gamma$ be a group, let $t$ and $n$ be positive integers, and let $(G, \gamma)$ be a $\Gamma$-labeled graph so that there exists a shifting $\gamma'$ of $\gamma$ and a tree-cut decomposition $(T, \mathscr{B})$ of $G$ such that for every bag $B \in \mathscr{B}$, at least one of the following holds.
    \begin{enumerate}
        \item There are at most $n$ vertices in the torso of $B$ of degree greater than $t$, and none of them are the new vertices of the torso.
        \item There exists $X \subseteq E(G)$ so that $(X,\gamma')$ is a certificate for $B$ of value at most~$t$.
    \end{enumerate}
    Then $G$ forbids a $(\Gamma, t + 1,n)$-rich flower immersion.
\end{lemma}
\begin{proof}
    Going for a contradiction, suppose that $G$ does admit such an immersion $\varphi$. 
    
    We claim that all of the branch vertices of $\varphi$ are in the same bag. Otherwise, there exist branch vertices $x$ and $y$ in different bags. Let $uv \in E(T)$ so that $x$ and $y$ come from different components of $T\setminus uv$. Note that there are at least $t+1$ pairwise edge-disjoint trails between $x$ and $y$. Now consider which outcome holds for the bag $B_u$. If the first outcome holds, then the torso of $B_u$ would have a new vertex of degree at least $t+1$ (obtained from the component of $T\setminus u$ containing $v$), a contradiction. If the second outcome holds, then the certificate for $B$ would have value at least $t+1$, again a contradiction.

    Thus the branch vertices must be entirely contained in some bag $B$. Note that each branch vertex must have degree at least $t+1$ in the torso. So, since there are $n+1$ branch vertices, the first outcome cannot hold for $B$. We can then delete at most $t$ edges to leave every component containing a branch vertex of $\varphi$ signed over the same proper subgroup of $\Gamma$. However, note that after deleting $t$ edges from even just a petal of the flower, it still contains a circuit of label $\alpha$ for each $\alpha \in \Gamma$. Moreover, we can choose these circuits to all have the center of the flower as their end. So shifting in $(G, \gamma)$ just conjugates the labels of the circuits, which only permutes the group elements. This contradicts the fact that each component containing a branch vertex of $\varphi$ is signed over a proper subgroup of $\Gamma$.
\end{proof}




Now, the first step towards proving the structure theorem is to use the local structure theorem (Theorem~\ref{thm:enrichingFlower}) to show that there is a certificate ``around'' every large flower immersion. More precisely, we will show that every ``large core'' has a certificate of bounded value. We refer the reader to Section~\ref{subsec:structureThm} for the definitions of certificates and their value.

A \textit{$t$-core} of a graph $G$ is a maximal set of vertices, each of degree at least $t+1$, which are pairwise $(t+1)$-edge-connected. (We say that two vertices $u$ and $v$ are \emph{$t$-edge-connected} if every edge cut separating $u$ and $v$ has size at least $t$. Note that being $t$-edge-connected is an equivalence relation.) We require the vertices to have degree at least $t+1$ only for the case where the $t$-core is a single vertex, and so the set of $t$-cores partitions the vertices with degree at least $t+1$. (Recall that the degree of a vertex does not count incident loops.) Each large $t$-core gives a flower immersion, as stated in the lemma below.

\begin{lemma}
\label{lem:TcoreGivesFlower}
    Let $k$ and $n$ be positive integers and let $G$ be a graph. Suppose that $S \subseteq V(G)$ is a set of more than $n$ vertices which are pairwise $kn$-edge-connected. Then $G$ admits a $(k,n)$-flower immersion whose branch vertices are contained in~$S$.
\end{lemma}
\begin{proof}
    The proof is as given in \cite[Lemma~1]{WOLLAN2015} with the obvious modification to deal with two separate variables $k$ and $n$.
\end{proof}

Now we use Theorem~\ref{thm:enrichingFlower} and Lemma~\ref{lem:TcoreGivesFlower} to prove that every large core has a certificate of small value.


\begin{lemma}
\label{lem:CoresHaveCertificate}
    For any finite group $\Gamma$ and positive integers $k$ and $n$, there is an integer $t=4kn|\Gamma|^{6+\lfloor\log_2|\Gamma|\rfloor}$ so that the following holds for any 2-edge-connected $\Gamma$-labeled graph $(G, \gamma)$ which forbids a $(\Gamma, k, n)$-rich flower immersion. For any $t$-core $S$ of size more than $n|\Gamma|$, there exists a set of vertices $B$ containing $S$ so that there is a certificate $(X, \gamma')$ for $B$ of value at most~$t$.
\end{lemma}
\begin{proof}
    For convenience, set $k' = k|\Gamma|^{4 + \lfloor\log_2 |\Gamma|\rfloor}$ and $n'=n|\Gamma|$. By Lemma~\ref{lem:TcoreGivesFlower}, there exists an immersion $\varphi$ of the $(k', n')$-flower into $G$ whose branch vertices are contained in $S$. (This step is slightly wasteful in order to make the next step more convenient.) By Theorem~\ref{thm:enrichingFlower}, there exist a shifting $\gamma'$ of $\gamma$ and a set $X \subseteq E(G)$ of size at most $2n'k'|\Gamma|$ such that the edge-block of $G \setminus X$ containing the branch center of $\varphi$ is labeled over a proper subgroup of $\Gamma$. Note that $2n'k'|\Gamma|=t/2$. Thus, since $S$ is a $t$-core and $t\geq |X|+2$, all of $S$ is contained in the same edge-block of~$G\setminus X$. 
    
    Let $B$ be the vertex set of this edge-block. Let $X'$ denote the union of $\delta_G(B)$ and $X \cap E(B)$. Thus $(X', \gamma')$ is a certificate for $B$. We claim that this certificate has value at most $2|X|$; this will complete the proof. First of all, notice that each edge in $X \cap E(B)$ contributes $2$ to the value. Next we show how to associate each edge $e \in \delta_G(B)\setminus X$ to an edge in $X \setminus E(B)$. So, let $e \in \delta_G(B)\setminus X$, and consider the component $B'$ of $(G\setminus X) \setminus B$ which contains an end of $e$. Since $G$ is $2$-edge-connected, $B'$ is incident to an edge $f \in X$, and we associate $e$ to $f$. Thus each edge $f \in X \cap \delta_G(B)$ has only one associated edge $e$, and the total contributed value of $e$ and $f$ is $2$. Likewise, each edge $f \in X$ which is not incident to $B$ has at most two associated edges $e$ and $e'$, and the total contributed value of $e$ and $e'$ is $2$. (In this case the edge $f$ itself does not contribute anything to the value.) This completes the proof.
\end{proof}

In order to use these certificates to obtain a tree-cut decomposition, we want to choose them in a special way. For the next lemma, it is useful to know the best value of a certificate for a given set of vertices. So let $\Gamma$ be a group, let $(G, \gamma)$ be a $\Gamma$-labeled graph, and let $B \subseteq V(G)$. Then the \textit{value} of $B$, denoted $\val(B)$, is the minimum value of a certificate for $B$. Notice that the value of $B$ equals $|\delta_G(B)|$ plus the minimum, over all shiftings $\gamma'$ of $\gamma$ and proper subgroups $\Gamma'$ of $\Gamma$, of twice the number of edges in $E(B)$ labeled outside of $\Gamma'$. The value is a \emph{posimodular} function on subsets of $V(G)$: for any $A,B \subseteq V(G)$,
$$\val(A) + \val(B) \geq \val(A \setminus B) + \val(B \setminus A).$$
This follows from the fact that $\delta_G(\cdot)$ is posimodular, and the minimum number of internal edges outside of a proper subgroup over all shiftings is increasing.

Now consider a $\Gamma$-labeled graph $(G, \gamma)$ and a family $\mathcal{S}$ of $t$-cores in $G$. A \emph{container system for $\mathcal{S}$} is a collection $\mathcal{C}$ of subsets of $V(G)$ so that each set in $\mathcal{S}$ is contained in a set in $\mathcal{C}$. The \emph{value} of $\mathcal{C}$ is the maximum value of a set in $\mathcal{C}$. Finally, we say that $\mathcal{C}$ is \emph{refined} if for any $C\in \mathcal{C}$, there exists a vertex $v \in C$ which is \emph{fully connected to $\delta(C)$}, that is, so that for every $A \subseteq C$ which contains $v$, we have $|\delta(A)|\geq |\delta(C)|$. Next we prove that there is a refined container system whose sets are pairwise disjoint.

\begin{lemma}
\label{lem:LaminarCertificates}
    Let $\Gamma$ be a group, let $t$ be a positive integer, let $(G, \gamma)$ be a $\Gamma$-labeled graph, and let $\mathcal{S}$ be a collection of $t$-cores which admits a container system of value at most~$t$. Then there exists a refined container system $\mathcal{C}$ for $\mathcal{S}$ of value at most~$t$ so that the sets in $\mathcal{C}$ are pairwise disjoint.
\end{lemma}
\begin{proof}
    Choose a container system $\mathcal{C}$ for $\mathcal{S}$ which has value at most $t$ and minimizes
    \begin{equation}\label{minFunction}
        \sum_{C \in \mathcal{C}} (|V(G)| + 1)^{\val(C)}\tag{$*$}
    \end{equation}
    and subject to this, minimizes
    \begin{equation}\label{secondaryMinimizer}
        \sum_{C \in \mathcal{C}} |C|. \tag{$**$}
    \end{equation} First note that there do not exist any distinct $B, C \in \mathcal{C}$ with $B \subseteq C$, for then we could just remove $B$ from $\mathcal{C}$.
    
    Now we claim that the sets in $\mathcal{C}$ are pairwise disjoint. Going for a contradiction, suppose that there exist distinct $B,C \in \mathcal{C}$ with non-empty intersection. By posimodularity, either $\val(B \setminus C) \leq \val(B)$ or $\val(C \setminus B) \leq \val(C)$. Suppose without loss of generality that $\val(B \setminus C) \leq \val(B)$. Let $\mathcal{C'}$ be obtained from $\mathcal{C}$ by removing $B$ and adding $B \setminus C$. We claim that $\mathcal{C}'$ is a container system for $\mathcal{S}$. To see this, let $S \in \mathcal{S}$ with $S \subseteq B$. Since the vertices in $S$ are pairwise $(t+1)$-edge-connected and $|\delta(C)|\leq \val(C)\leq t$, either $S \subseteq C$ or $S \subseteq B \setminus C$. Either way, there is a set in $\mathcal{C}'$ which contains $S$. Finally, note that \eqref{minFunction} does not increase since $\val(B \setminus C) \leq \val(B)$, and \eqref{secondaryMinimizer} strictly decreases since $|B\setminus C|<|B|$. This contradicts the choice of~$\mathcal{C}$.

    It just remains to prove that $\mathcal{C}$ is refined. Going for a contradiction, suppose that there exists $C \in \mathcal{C}$ such that no vertex in $C$ is fully connected to $\delta(C)$. Then, in particular, for each $S \in \mathcal{S}$ which is contained in $C$, there exist a vertex $v_S \in S$ and a set $A_S \subseteq C$ so that $v_S \in A_S$ and $|\delta(A_S)|<|\delta(C)|$. (We may choose $v_S$ to be any arbitrary vertex in $S$.) Since $|\delta(C)| \leq \val(C)\leq t$ and the vertices in $S$ are pairwise $(t+1)$-edge-connected, we actually have $S \subseteq A_S$. Thus we can obtain a new container system $\mathcal{C}'$ from $\mathcal{C}$ by removing $C$ and adding all of the sets $A_S$. 
    
    Note that each set $A_S$ has value strictly less than the value of $C$; the minimum number of edges outside of a subgroup over all shiftings does not increase, and $|\delta(A_S)|<|\delta(C)|$. Since the $t$-cores of a graph are pairwise vertex-disjoint, $|\mathcal{S}|\leq |V(G)|$ and there are at most $|V(G)|$ new terms in the sum in \eqref{minFunction}. Moreover, each term is smaller by a factor of at least $|V(G)| + 1$. This contradicts the choice of $\mathcal{C}$ and completes the proof of Lemma~\ref{lem:LaminarCertificates}.
\end{proof}

We now prove the following lemma, which nearly completes the proof of Theorem~\ref{thm:MainStructureThm}. Since this lemma works relative to a fixed container system, we do not need to consider a group labeling.


\begin{lemma}
\label{lem:decompFromContainers}
    Let $t$ and $n$ be positive integers, let $G$ be a $2$-edge-connected graph, and let $\mathcal{C}$ be a refined container system for the collection of all $t$-cores of size more than $n$ so that the sets in $\mathcal{C}$ are pairwise disjoint. Then there exists a tree-cut decomposition $(T, \mathscr{B})$ of $G$ such that for every bag $B \in \mathscr{B}$, at least one of the following holds.
    \begin{enumerate}
        \item\label{itm:bagType1} There are at most $n$ vertices in the torso of $B$ of degree greater than $t$, and none of them are the new vertices of the torso.
        \item\label{itm:bagType2} The set $B$ is in $\mathcal{C}$.
    \end{enumerate}
\end{lemma}
\begin{proof}
    We proceed by induction on $|V(G)|$. We may assume that there exists a $t$-core since otherwise all vertices have degree at most $t$ and we can take the tree-cut decomposition with one bag. Let $S$ be a $t$-core of maximum size. 
    
    Define a set of vertices $C$ as follows. If $|S|>n$, then let $C$ be the set in $\mathcal{C}$ containing $S$. Otherwise, let $C$ be a minimal set of vertices which contains $S$ and has $|\delta(C)|\leq t$. (Such a set exists since we could take $C=V(G)$.) Note that in either case, there exists a vertex in $C$ which is fully connected to $\delta(C)$. In the first case, such a vertex exists since $\mathcal{C}$ is refined. In the second case, such a vertex exists by the minimality of $C$ and since the vertices in $S$ are pairwise $(t+1)$-edge-connected. 
   
    Now, let $G'$ be the graph formed from $G$ by identifying $C$ to a single vertex; we do not include any loops at this new vertex. Notice that $G'$ is $2$-edge-connected. Also note that $G'$ has strictly fewer vertices than $G$; if $|S|>1$ then this is clear, and if $|S|=1$ then $C \neq S$ since the vertex in $S$ has degree at least $t+1$ (not counting loops) and $|\delta(C)|\leq t$. Moreover, notice that $G$ admits $G'$ as an immersion since $C$ contains a vertex which is fully connected to $\delta(C)$. It follows that the edge connectivity between any two vertices in $V(G') \cap V(G)$ is at least as large in $G$ as in $G'$. Thus, since the new vertex of $G'$ has degree at most $t$, every $t$-core of $G'$ is contained in a $t$-core of $G$.
    
    Now, if $|S|>n$, then let $\mathcal{C}'$ be the collection obtained from $\mathcal{C}$ by removing $C$, and otherwise let $\mathcal{C}'$ be empty. Note that $\mathcal{C}'$ is a refined container system for the collection of all $t$-cores of $G'$ of size more than $n$, and the sets in $\mathcal{C}'$ are pairwise disjoint. Thus we can apply induction to obtain a tree-cut decomposition $(T', \mathscr{B}')$ of $G'$. Let $(T, \mathscr{B})$ be the tree-cut decomposition of $G$ obtained from $(T', \mathscr{B}')$ by removing the new vertex from the bag $B_u' \in \mathscr{B}'$ which contains it, and adding a new leaf node to $T'$ at the vertex $u$ so that the bag of this new leaf is $C$. 
    
    Finally, consider a bag $B \in \mathscr{B}$. If $B=C$ and $|S|>n$, then $B \in \mathcal{C}$ and the second outcome of the lemma holds. If $B=C$ and $|S|\leq n$, then the new vertex of the torso of $B$ has degree at most $t$ since $|\delta(C)| \leq t$. Moreover, since $C$ was chosen to be minimal, no vertex in $C\setminus S$ has degree larger than $t$ because there is some edge cut of size at most $t$ which separates this vertex from $S$. Thus the first outcome holds for $B$. So we may assume that $B\neq C$. 
    
    Thus there exists a bag $B' \in \mathscr{B}'$ which corresponds to $B$. If $B' \in \mathcal{C}'$, then $B'$ does not contain the new vertex of $G'$, and so $B=B'$ and the second outcome of the lemma holds. Thus we may assume that there are at most $n$ vertices in the torso of $B'$ of degree greater than $t$, and none of them are the new vertices of the torso. If $B'$ does not contain the new vertex of $G'$, then $B=B'$ and the torso of $B$ in $(T, \mathscr{B})$ is the same as the torso of $B'$ in $(T', \mathscr{B}')$. Thus in this case the first outcome of the lemma holds. So we may assume that $B'$ contains the new vertex of $G'$. This case is also fine since the new leaf vertex of $T$ will yield a new vertex in the torso of $B$ of degree at most $t$. This completes the proof of Lemma~\ref{lem:decompFromContainers}.
\end{proof}

We are ready to prove the structure theorem, which we restate for convenience.

\maintheorem*

\begin{proof}
    Recall that the $(\Gamma, k, n)$-rich flower admits as an immersion any $\Gamma$-labeled graph with $n$ vertices where every vertex is incident to at most $k$ edges. Thus $(G, \gamma)$ forbids a $(\Gamma, k, n)$-rich flower immersion. So by Lemma~\ref{lem:CoresHaveCertificate}, every $t$-core of size more than $n|\Gamma|$ is contained in a set of value at most $t$. Let $\mathcal{S}$ be the set of all $t$-cores of size more than $n|\Gamma|$. Now by Lemma~\ref{lem:LaminarCertificates}, there exists a refined container system $\mathcal{C}$ for $\mathcal{S}$ so that $\mathcal{C}$ has value at most $t$ and the sets in $\mathcal{C}$ are pairwise disjoint. Let $\gamma'$ be a shifting of $\gamma$ so that for each $C \in \mathcal{C}$, there exists a set of edges $X$ so that $(X, \gamma')$ is a certificate for $C$ of value at most $t$. We can obtain $\gamma'$ by shifting for each $C \in \mathcal{C}$ in any arbitrary order; note that shifting at a vertex outside of $C$ will not affect the certificate $(X, \gamma')$ for $C$. Now, let $(T, \mathscr{B})$ be the tree-cut decomposition of $G$ obtained by applying Lemma~\ref{lem:decompFromContainers} to $\mathcal{C}$ and the $t$-cores of size more than $n|\Gamma|$. Then $\gamma'$ and $(T, \mathscr{B})$ satisfy the desired conditions.
\end{proof}

\section{Acknowledgments}
Part of this work was completed at the 2026 Barbados Graph Theory Workshop held at the Bellairs Research Institute in January 2026. We are grateful to the institute for providing an excellent research environment. 

\bibliography{refs}

\end{document}